\documentclass[oneside, 11pt, reqno]{amsart}

\usepackage[utf8]{inputenc}

\usepackage{geometry}
\usepackage[shortlabels]{enumitem}
\usepackage{amsmath}
\usepackage{amsthm,amssymb}
\usepackage{mathtools,mathrsfs}
\usepackage{float}
\usepackage{hyperref}
\usepackage[nosort,capitalize]{cleveref}
\usepackage{xspace}
\usepackage{booktabs}
\usepackage{etoolbox}
\usepackage{tikz}
\usepackage{tikz-cd}

\newcommand{\bibsortkey}[1]{}

\usetikzlibrary{decorations.pathmorphing}
\tikzset{symbol/.style={draw=none,every to/.append style={
      edge node={node [sloped, allow upside down, auto=false]{$#1$}}}}}
\newcommand{\ardual}[1]{\ar[#1, <->,squiggly]}

\newcommand{\arrestrict}{\ar[u, symbol=\leqslant]}

\newcommand{\areqv}[1]{\ar[#1, <->]}
\newcommand{\arrestrictright}{\ar[r, symbol=\geqslant]}

\newtheorem{theorem}{Theorem}[section]
\newtheorem{lemma}[theorem]{Lemma}

\newtheorem{corollary}[theorem]{Corollary}

\theoremstyle{definition}
\newtheorem{definition}[theorem]{Definition}
\newtheorem{remark}[theorem]{Remark}

\newcommand{\upset}{\mathord\uparrow}
\newcommand{\downset}{\mathord\downarrow}
\newcommand{\cl}{\mathop{\sf cl}}
\renewcommand{\int}{\mathop{\sf int}}
\DeclareMathOperator{\reg}{reg}
\DeclareMathOperator{\core}{core}
\DeclareMathOperator{\cen}{cen}

\newcommand{\cat}[1]{{\sf #1}\xspace}
\newcommand{\Top}{\cat{Top}}
\newcommand{\Sob}{\cat{Sob}}
\newcommand{\LCSob}{\cat{LKSob}}
\newcommand{\SLCSp}{\cat{StLKSp}}
\newcommand{\SKSp}{\cat{StKSp}}
\newcommand{\KHaus}{\cat{KHaus}}
\newcommand{\KBSob}{\cat{KBSob}}
\newcommand{\SKBSp}{\cat{StKBSp}}
\newcommand{\Spec}{\cat{Spec}}
\newcommand{\Stone}{\cat{Stone}}

\newcommand{\Frm}{\cat{Frm}}
\newcommand{\SFrm}{\cat{SFrm}}
\newcommand{\CFrm}{\cat{ConFrm}}
\newcommand{\SCFrm}{\cat{StCFrm}}
\newcommand{\SKFrm}{\cat{StKFrm}}
\newcommand{\KRFrm}{\cat{KRFrm}}
\newcommand{\AlgFrm}{\cat{AlgFrm}}
\newcommand{\ArithFrm}{\cat{AriFrm}}
\newcommand{\CohFrm}{\cat{CohFrm}}
\newcommand{\StoneFrm}{\cat{StoneFrm}}

\newcommand{\LPries}{\cat{LPries}}
\newcommand{\SL}{\cat{SLPries}}
\newcommand{\CL}{\cat{ConLPries}}
\newcommand{\SCL}{\cat{StCLPries}}
\newcommand{\SKL}{\cat{StKLPries}}
\newcommand{\KRL}{\cat{KRLPries}}
\newcommand{\AlgL}{\cat{AlgLPries}}
\newcommand{\ArithL}{\cat{AriLPries}}
\newcommand{\CohL}{\cat{CohLPries}}
\newcommand{\StoneL}{\cat{StoneLPries}}

\newcommand{\functor}[1]{\mathscr #1}
\newcommand{\clopup}{{\sf ClopUp}}
\newcommand{\clopbi}{{\sf ClopBi}}
\newcommand{\clopsup}{{\sf ClopSUp}}

\crefname{figure}{Figure}{Figures}

\setlist[enumerate,1]{label={\upshape(\arabic*)}}

\title{Algebraic Frames in Priestley duality}
\makeatletter
\@namedef{subjclassname@2020}{
  \textup{2020} Mathematics Subject Classification}
\patchcmd{\@setaddresses}{\indent}{\noindent}{}{}
\patchcmd{\@setaddresses}{\indent}{\noindent}{}{}
\patchcmd{\@setaddresses}{\indent}{\noindent}{}{}
\patchcmd{\@setaddresses}{\indent}{\noindent}{}{}
\makeatother



\author{G.~Bezhanishvili and S.~Melzer}
\address{\newline
Department of Mathematical Sciences\newline
New Mexico State University\newline
Las Cruces, NM 88003\newline
USA\newline}
\email{guram@nmsu.edu}
\email{smelzer@nmsu.edu}

\makeatletter
\newcommand{\plabel}[1]
{
    \label{#1}
    \edef\curlabel{#1}
    \edef\curenv{\@currenvir}
}
\makeatother
\newcommand{\clabel}[1]
{   
    \label[\curenv]{\curlabel-#1}
}

\begin{document}

\begin{abstract}
    We characterize Priestley spaces of algebraic, arithmetic, coherent, and Stone frames. As a corollary, we derive the well-known dual equivalences in pointfree topology involving various categories of algebraic frames.
\end{abstract}

\maketitle

\tableofcontents

\section{Introduction} \label{sec:intro}

A complete lattice is algebraic provided every element is a join of compact elements. 
Algebraic lattices arise naturally in different contexts. For example, 
the lattice of subalgebras as well as the lattice of congruences of any algebra is algebraic, and up to isomorphism, every algebraic lattice arises this way (see, e.g., \cite{BS81}). It is a well-known result of Nachbin \cite{Nachbin1949} (see also \cite{BirkhoffFrink1948}) that algebraic lattices are exactly the ideal lattices of join-semilattices. If an algebraic lattice $L$ is distributive, then the infinite distributive law $a\wedge\bigvee S=\bigvee\{a\wedge s\mid s\in S\}$ holds, and hence $L$ is a frame. Such frames are known as algebraic frames and have been the subject of study in pointfree topology and domain theory (see, e.g., \cite{Compendium2003,PicadoPultr2012}).

There is a well-developed duality theory for the category \AlgFrm of algebraic frames and its various subcategories 
such as the categories of arithmetic frames (also known as M-frames), coherent frames, and Stone frames. Indeed, a frame $L$ is algebraic iff it is the frame of opens of a compactly based sober space $X$ \cite[p.~423]{Compendium2003}.
In addition, $L$ is arithmetic iff $X$ is stably compactly based, $L$ is coherent iff $X$ is spectral, and $L$ is a Stone frame iff $X$ is a Stone space (see \cref{sec: prelims} for details). 

The duality theory for algebraic frames is a restriction of the well-known Hofmann--Lawson duality \cite{HofmannLawson1978}. We recall (see, e.g., \cite[p.~135]{PicadoPultr2012}) that a frame $L$ is {\em continuous} if the way-below relation $\ll$ is approximating. In addition, $L$ is {\em stably continuous} if $\ll$ is stable ($a\ll b,c$ implies $a\ll b\wedge c$), $L$ is {\em stably compact} if moreover $L$ is compact, and $L$ is {\em compact regular} if furthermore $\ll$ coincides with the well-inside relation $\prec$. We thus obtain the following correspondence between various categories of continuous and algebraic frames, where the categories are defined in \cref{table:con frames,table:alg frames} and $\leqslant$ stands for being a full subcategory of.
\
\begin{figure}[H] 
\centering
\begin{tikzcd}[ampersand replacement=\&, column sep={6em,between origins}, row sep=2em]
    \CFrm \ar[r, symbol=\geqslant] 
        \& \SCFrm \ar[r, symbol=\geqslant]
        \& \SKFrm \ar[r, symbol=\geqslant]
        \& \KRFrm \\
    \AlgFrm \arrestrict \ar[r, symbol=\geqslant]
        \& \ArithFrm \arrestrict \ar[r, symbol=\geqslant]
        \& \CohFrm \arrestrict \ar[r, symbol=\geqslant]
        \& \StoneFrm \arrestrict
\end{tikzcd} 
\caption{Inclusion relationships between categories of continuous and algebraic frames. \label{diagram: CFrmAlgFrm}}
\end{figure}

By the well-known Priestley duality \cite{Priestley1970,Priestley1972}, the category of bounded distributive lattices is dually equivalent to the category of Priestley spaces. 
Pultr and Sichler \cite{PultrSichler1988} provided a restricted version of Priestley duality for the category \Frm of frames and frame homomorphisms. This line of research was further developed by several authors (see, e.g.,  \cite{PultrSichler2000,BezhGhilardi2007,BezhanishviliGabelaiaJibladze2013,BezhanishviliGabelaiaJibladze2016,AvilaBezhanishviliMorandiZaldivar2020,AvilaBezhanishviliMorandiZaldivar2021}). In \cite{BezhanishviliMelzer2022b}, we obtained Priestley duality for \CFrm and its subcategories listed in the first row of Figure~\ref{diagram: CFrmAlgFrm}. The resulting (dual) equivalences are outlined in \cref{diagram: equivalences}.
The aim of this paper is to further study Priestley duality for \AlgFrm and its subcategories listed in the second row of \cref{diagram: CFrmAlgFrm}. This in requires characterizing Priestley spaces of algebraic, coherent, arithmetic, and Stone frames. 

The paper is organized as follows. 
In \cref{sec: prelims}, we describe the above
categories of continuous and algebraic frames, as well as the corresponding categories of locally compact and compactly based sober spaces. \cref{sec 3} recalls Priestley duality for
various categories of continuous frames. 
In \cref{sec: alg frm}, we characterize
Priestley spaces of algebraic frames. Consequently, we obtain a new proof of the duality between \AlgFrm and \KBSob. Finally, in \cref{sec 5}, we 
characterize Priestley spaces of arithmetic, coherent, and Stone frames. In each case, this yields a new proof of the 
duality between the 
corresponding 
categories of algebraic frames and compactly based spaces. We conclude the paper by connecting Priestley spaces of coherent frames and Stone frames to Priestley duality for bounded distributive lattices and Stone duality for boolean algebras.
\color{black}

\section{Continuous and algebraic frames} \label{sec: prelims}
A \emph{frame} is a complete lattice $L$ satisfying the join-infinite distributive law
\[
    a \wedge \bigvee S = \bigvee\{a \wedge s \mid s \in S\}
\]
for every $a \in L$ and $S \subseteq L$. A \emph{frame homomorphism} is a map between frames that preserves finite meets and arbitrary joins. Let \Frm be the category of frames and frame homomorphisms. 
A frame is \emph{spatial} if completely prime filters separate elements of $L$. 
Let \SFrm be the full subcategory of \Frm consisting of spatial frames. 

As usual, we write $\ll$ for the {\em way below} relation in a frame $L$ and recall that $a\ll b$ provided for each $S\subseteq L$ we have $b\le\bigvee S$ implies $a\le \bigvee T$ for some finite $T\subseteq S$. We call $a\in L$ \emph{compact} if $a \ll a$ and $L$ \emph{compact} if its top element is compact. We write $K(L)$ for the collection of compact elements of $L$.

In the introduction we recalled the definitions of continuous, stably continuous, stably compact, and compact regular frames. A frame homomorphism $h : L \to M$ between continuous frames is \emph{proper} if it preserves $\ll$; that is, $a \ll b$ implies $h(a) \ll h(b)$ for all $a,b \in L$. Let \CFrm be the category of continuous frames and proper frame homomorphisms. We write \SCFrm and \SKFrm for the full subcategories of \CFrm consisting of stably continuous and stably compact frames, respectively. We also let \KRFrm be the full subcategory of \Frm consisting of compact regular frames. Since every frame homomorphism between compact regular frames is proper, \KRFrm is a full subcategory of \SKFrm. 

We have the following categories of continuous frames.

\begin{table}[H]
\begin{tabular}{lll}
    \toprule 
    \multicolumn{1}{c}{\bf Category} & \multicolumn{1}{c}{\bf Objects} & \multicolumn{1}{c}{\bf Morphisms} \\ \midrule
    \CFrm  & continuous frames        & proper frame homomorphisms \\
    \SCFrm & stably continuous frames & proper frame homomorphisms \\
    \SKFrm & stably compact frames    & proper frame homomorphisms \\ 
    \KRFrm & compact regular frames   & frame homomorphisms \\
    \bottomrule
\end{tabular}
\caption{Categories of continuous frames.\label{table:con frames}}
\end{table}

\begin{definition}
\begin{enumerate}
    \item[]
    \item (\cite[p.~142]{PicadoPultr2012}) A frame $L$ is \emph{algebraic} if $a = \bigvee\{b \in K(L) \mid b \leq a\}$ for all $a \in L$.
    \item (\cite[p.~64]{Johnstone1982}) A frame homomorphism $h : L \to M$ is \emph{coherent} if $a \in K(L)$ implies $h(a) \in K(M)$.
    \item Let \AlgFrm be the category of algebraic frames and coherent frame homomorphisms.
\end{enumerate} 
\end{definition}

\begin{remark} \label{rem: full sub alg}
    It is easy to see that every algebraic frame is continuous, and that
    a frame homomorphism between coherent frames is coherent iff it is proper.  Consequently, \AlgFrm is a full subcategory of \CFrm.
\end{remark}

\begin{definition}
\begin{enumerate}
    \item[]
    \item A frame $L$ is \emph{arithmetic} if it is algebraic and $\ll$ is stable.
    \item Let \ArithFrm be the full subcategory of \AlgFrm consisting of arithmetic frames. 
\end{enumerate}
\end{definition}

\begin{remark}
    \begin{enumerate}
        \item[]
        \item 
        In \cite{Compendium2003} a lattice is called arithmetic if the binary meet of compact elements is compact. For algebraic lattices this is equivalent to $\ll$ being stable (see, e.g. \cite[Proposition I-4.8]{Compendium2003}).
        \item Arithmetic frames 
        are also called M-frames; see, e.g.,~\cite{IberkleidMcGovern2009,Bhattacharjee2018}.
    \end{enumerate}
\end{remark}

\begin{definition}
    \begin{enumerate}
    \item[]
    \item (\cite[p.~63--64]{Johnstone1982}) A frame $L$ is \emph{coherent} if $L$ is arithmetic and compact.
    \item Let \CohFrm be the full subcategory of \ArithFrm consisting of coherent frames.
    \end{enumerate}
\end{definition}

Let $L$ be a frame. We recall that the \emph{well-inside} relation on $L$ is defined by $a \prec b$ if $a^* \vee b = 1$, where $a^*:=\bigvee \{x\in L\mid a\wedge x=0\}$ is the pseudocomplement of $a$. An element $a\in L$ is \emph{complemented} if $a \prec a$. Let $C(L)$ be the collection of  complemented elements of $L$. 
It is well known that if $L$ is compact, then $a\prec b$ implies $a\ll b$; and if $L$ is regular, then $a\ll b$ implies $a\prec b$. Thus, in compact regular frames, the two relations $\ll$ and $\prec$ coincide, and hence $K(L)=C(L)$. 

The next definition is well known (see, e.g., \cite{Johnstone1982,Banaschewski1989,Jakl2013}). We thank Joanne Walters-Wayland for pointing out to us that the terminology of Stone frames originated from Banaschewski's University of Cape Town lecture notes (1988).

\begin{definition}
    \begin{enumerate}
        \item[] 
        \item 
        A frame $L$ is \emph{zero-dimensional} if $a = \bigvee \{b \in C(L) \mid b \leq a\}$ for all $a \in L$.
        \item 
        A {\em Stone frame} is a compact zero-dimensional frame.
        \item Let \StoneFrm be the full subcategory of \Frm consisting of Stone frames.
    \end{enumerate}
\end{definition}

\begin{remark} \label{rem: full sub stonefrm}
    Clearly \StoneFrm is a full subcategory of \KRFrm.
    Moreover, since every frame homomorphism preserves $\prec$ and in Stone frames $\prec$ coincides with $\ll$, we have that \StoneFrm is a full subcategory of \CohFrm.
\end{remark}

We have the following categories of algebraic frames.

\begin{table}[H]
\begin{tabular}{lll}
    \toprule 
    \multicolumn{1}{c}{\bf Category} & \multicolumn{1}{c}{\bf Objects} & \multicolumn{1}{c}{\bf Morphisms} \\ \midrule
    \AlgFrm  & algebraic frames       & coherent frame homomorphisms \\
    \ArithFrm  & arithmetic frames    & coherent frame homomorphisms \\
    \CohFrm  & coherent frames        & coherent frame homomorphisms \\
    \StoneFrm & Stone frames          & frame homomorphisms \\
    \bottomrule
\end{tabular}
\caption{Categories of algebraic frames.\label{table:alg frames}}
\end{table}
The categories of algebraic and continuous frames relate to each other as shown in \cref{diagram: CFrmAlgFrm}.
We next turn our attention to the corresponding categories of topological spaces. The following definitions are well known (see, e.g., \cite[pp.~43--44]{Compendium2003}). A closed subset of a topological space $X$ is {\em irreducible} if it cannot be written as the union of two proper subsets. We call $X$ \emph{sober} if each irreducible closed subset is the closure of a unique point in $X$, and \emph{locally compact} if for every open set $U$ and $x \in U$ there are an open set $V$ and a compact set $K$ such that $x \in V \subseteq K \subseteq U$. 

Following \cite[Lemma VI-6.21]{Compendium2003}, we call a continuous map $f : X \to Y$ 
between locally compact sober spaces
\emph{proper} if $f^{-1}(K)$ is compact for each compact saturated set $K \subseteq Y$.
Let \LCSob be the category of locally compact sober spaces and proper maps between them. 

A topological space $X$ is \emph{coherent} if the intersection of two compact saturated sets is compact (\cite[p.~474]{Compendium2003}), and $X$ is \emph{stably locally compact} if it is locally compact, sober, and coherent. Let \SLCSp be the full subcategory of \LCSob consisting of stably locally compact spaces. 

A compact stably locally compact space is a \emph{stably compact} space (\cite[p.~476]{Compendium2003}). We write \SKSp for the full subcategory of \SLCSp consisting of stably compact spaces. Also, we denote by \KHaus the category of compact Hausdorff spaces and continuous maps. Since a continuous map between compact Hausdorff spaces is proper, \KHaus is a full subcategory of \SKSp.

We have the following categories of locally compact sober spaces.

\begin{table}[H]
\begin{tabular}{lll}
    \toprule 
    \multicolumn{1}{c}{\bf Category} & \multicolumn{1}{c}{\bf Objects} & \multicolumn{1}{c}{\bf Morphisms} \\ \midrule
    \LCSob & locally compact sober spaces  & proper maps \\
    \SLCSp & stably locally compact spaces & proper maps \\
    \SKSp  & stably compact spaces         & proper maps \\ 
    \KHaus & compact Hausdorff spaces      & continuous maps\\
    \bottomrule
\end{tabular}
\caption{Categories of locally compact sober spaces.\label{table:lcsob spaces}}
\end{table}

We now shift our focus to compactly based spaces. We recall that a continuous map $f : X \to Y$ 
is {\em coherent} if $f^{-1}(U)$ is compact for each compact open $U \subseteq Y$.

\begin{definition} 
\begin{enumerate}[ref=\thedefinition(\arabic*)]
\item[] 
\item (\cite[p.~2063]{Erne2009}) A topological space $X$ is \emph{compactly based} if it has a basis of compact open sets. Let \KBSob be the category of compactly based sober spaces and coherent maps. 
\item A compactly based space $X$ is \emph{stably compactly based} if it is sober and the intersection of two compact opens is compact. Let \SKBSp be the full subcategory of \KBSob consisting of stably compactly based spaces.
\item (\cite[p.~43]{Hochster1969}) A stably compactly based space $X$ is a \emph{spectral} space if it is compact. Let \Spec be the full subcategory of \SKBSp consisting of spectral spaces.
\item (\cite[p.~70]{Johnstone1982}) A \emph{Stone} space is a zero-dimensional compact Hausdorff space. Let \Stone be the category of Stone spaces and continuous maps.
\end{enumerate}
\end{definition}

We have the following categories of compactly based sober spaces.
 
\begin{table}[H]
\begin{tabular}{lll}
    \toprule 
    \multicolumn{1}{c}{\bf Category} & \multicolumn{1}{c}{\bf Objects} & \multicolumn{1}{c}{\bf Morphisms} \\ \midrule
    \KBSob & compactly based sober spaces & coherent maps\\
    \SKBSp & stably compactly based spaces & coherent maps \\
    \Spec & spectral spaces                & coherent maps\\
    \Stone & Stone spaces                  & continuous maps \\
    \bottomrule
\end{tabular}
\caption{Categories of compactly based sober spaces.\label{table:kb spaces}}
\end{table}

\begin{remark} \label{rem: comp sat} \label{rem: full sub kbsob} \label{rem: full sub stone}
It is 
easy to see that \Stone is a full subcategory of \Spec (see, e.g., \cite[p.~71]{Johnstone1982}). To see that \KBSob is a full subcategory of \LCSob, it is sufficient to observe that a continuous map between compactly based sober spaces is coherent iff it is proper. For this it is enough to observe that
in a compactly based space $X$, every compact saturated set is an intersection of compact opens. 
To see this, let $K \subseteq X$ be compact saturated. It suffices to show that for each $x \not \in K$ there is a compact open $U$ containing $K$ and missing $x$. 
For each $y \in K$ there is a compact open $U_y$ such that $y \in U_y$ and $x \not \in U_y$. Therefore, $K \subseteq \bigcup \{U_y \mid y \in K\}$. By compactness of $K$ and the fact that a finite union of compact sets is compact, there is a compact open $U$ such that $K \subseteq U$ and $x \not \in U$.
\end{remark}

We thus obtain the following correspondence between various categories of locally compact and compactly based sober spaces.
\begin{figure}[H] 
\centering
\begin{tikzcd}[ampersand replacement=\&, column sep={6em,between origins}, row sep=2em]
    \LCSob \ar[r, symbol=\geqslant] 
        \& \SLCSp \ar[r, symbol=\geqslant]
        \& \SKSp \ar[r, symbol=\geqslant]
        \& \KHaus \\
    \KBSob \arrestrict \ar[r, symbol=\geqslant]
        \& \SKBSp \arrestrict \ar[r, symbol=\geqslant]
        \& \Spec \arrestrict \ar[r, symbol=\geqslant]
        \& \Stone \arrestrict
\end{tikzcd} 
\caption{Inclusion relationships between categories of locally compact and compactly based sober spaces. \label{diagram: LCSobKBSob}}
\end{figure}

There is a well-known dual adjunction between \Top and \Frm, 
which restricts to a dual equivalence between \Sob and \SFrm (see, e.g., \cite[Section II-1]{Johnstone1982}). Further restrictions of this equivalence yield the following well-known duality results for continuous frames:

\begin{theorem} \plabel{cfrm-dualities}
    \begin{enumerate}[ref=\thetheorem(\arabic*)]
        \item[]
        \item 
        \CFrm is dually equivalent to \LCSob. \clabel{1}
        \item \SCFrm is dually equivalent to \SLCSp. \clabel{2}
        \item \SKFrm is dually equivalent to \SKSp. \clabel{3}
        \item 
        \KRFrm is dually equivalent to \KHaus. \clabel{4}
    \end{enumerate}
\end{theorem}
We thus arrive at the following diagram, where $\leftrightsquigarrow$ represents dual equivalence.
\begin{figure}[H]
\begin{center}
\begin{tikzcd}[ampersand replacement=\&, column sep={6em,between origins}, row sep=2em]
    \CFrm \ardual{d} \arrestrictright 
        \& \SCFrm \ardual{d} \arrestrictright 
        \& \SKFrm \ardual{d} \arrestrictright 
        \& \KRFrm \ardual{d}\\
    \LCSob \arrestrictright 
        \& \SLCSp \arrestrictright 
        \& \SKSp \arrestrictright 
        \& \KHaus
\end{tikzcd}
\end{center}
\caption{Correspondence between categories of continuous frames and locally compact spaces.} 
\end{figure}

\begin{remark}
\cref{cfrm-dualities-1} is known as Hofmann--Lawson duality \cite{HofmannLawson1978} (see also \cite[Proposition V-5.20]{Compendium2003}). 
The origins of \cref{cfrm-dualities-2,cfrm-dualities-3} can be traced back to \cite{GierzKeimel1977,Johnstone1981,Simmons1982,Banaschewski1981} (see also \cite[Section VI-7.4]{Compendium2003}). Finally, \cref{cfrm-dualities-4} is known as Isbell duality \cite{Isbell1972} (see also \cite{BanaschweskiMulvey1980} or \cite[Section VII-4]{Johnstone1982}).
\end{remark}

We next describe the duality results for algebraic frames. One of the earliest references is probably
\cite[Theorem 5.7]{HofmannKeimel1972} (see also \cite[p.~423]{Compendium2003}), 
where the dualities for \AlgFrm, \ArithFrm, and \CohFrm are stated. The duality for \CohFrm is also described in 
\cite{Banaschewski1980,Banaschewski1981}. This further reduces to the duality for \StoneFrm (see, e.g., 
\cite[Chapter IV]{Jakl2013}). 

\begin{theorem} \plabel{frm-dualities}
    \begin{enumerate}[ref=\thetheorem(\arabic*)]
        \item[]
        \item \AlgFrm is dually equivalent to \KBSob. \clabel{alg}
        \item \ArithFrm is dually equivalent to \SKBSp. \clabel{arith}
        \item \CohFrm is dually equivalent to \Spec. \clabel{coh}
        \item \StoneFrm is dually equivalent to \Stone. \clabel{stone}
    \end{enumerate}
\end{theorem}

We thus arrive at the following diagram.

\begin{figure}[H]
\begin{center}
\begin{tikzcd}[ampersand replacement=\&, column sep={6em,between origins}, row sep=2em]
    \AlgFrm \ardual{d} \arrestrictright 
        \& \ArithFrm \ardual{d} \arrestrictright 
        \& \CohFrm \ardual{d} \arrestrictright 
        \& \StoneFrm \ardual{d}\\
    \KBSob \arrestrictright 
        \& \SKBSp \arrestrictright 
        \& \Spec \arrestrictright 
        \& \Stone
\end{tikzcd}
\end{center}
\caption{Correspondence between categories of algebraic frames and compactly based spaces.} \label{diagram: introduction}
\end{figure}

\begin{remark}
    The proof of \cref{frm-dualities} can easily be deduced from \cref{cfrm-dualities} and the fact that \AlgFrm and \KBSob are full subcategories of \CFrm and \LCSob, respectively. But it is easy to give a direct proof of \cref{frm-dualities} which does not rely on \cref{cfrm-dualities}. For this it is sufficient to observe that every algebraic frame is spatial. 
    Let $L$ be an algebraic frame. 
    Then Scott-open filters separate elements of $L$. To see this, if $a\not\le b$, then there is $k\in K(L)$ such that $k\le a$ but $k\not\le b$. Thus, $\upset k$ is a Scott-open filter containing $a$ and missing $b$. It is left to observe that the Prime Ideal Theorem implies that $L$ is spatial iff Scott-open filters separate elements of $L$ (see, e.g., \cite[Corollary 5.9(2)]{BezhanishviliMelzer2022}).
\end{remark}

\section{Priestley duality for continuous frames} \label{sec 3}

As we pointed out in the Introduction, Pultr and Sichler \cite{PultrSichler1988} restricted Priestley duality
for bounded distributive lattices to the category of frames. 
In this section we briefly recall Pultr--Sichler duality and its restriction to various categories of continuous frames. 

A \emph{Priestley space} is a Stone space $X$ with a partial order $\leq$ such that clopen upsets separate points. 
An \emph{Esakia space} is a Priestley space with the additional property that the partial order $\le$ is continuous (the downset of each clopen is clopen). By Esakia duality \cite{Esakia1974}, Esakia spaces are exactly the Priestley spaces of Heyting algebras. An important feature of Esakia spaces is that the closure of each upset is an upset. Esakia duals of complete Heyting algebras have the additional property that the closure of each open upset is open. Such Esakia spaces are called \emph{extremally order-disconnected} as they generalize extremally disconnected Stone spaces. Since frames are complete Heyting algebras, Priestley duals of frames are exactly the extremally order-disconnected Esakia spaces.
\begin{definition}
    \begin{enumerate}
        \item[]
        \item An \emph{L-space} (\emph{localic space}) is an extremally order-disconnected Esakia space.
        \item An \emph{L-morphism} is a continuous order-preserving map $f : X \to Y$ between L-spaces such that $f^{-1}\cl U = \cl f^{-1}U$ for every open upset $U$ of $Y$.
        \item Let \LPries be the category of L-spaces and L-morphisms.
    \end{enumerate}
\end{definition}

\begin{theorem}[Pultr--Sichler {\cite[Corollary 2.5]{PultrSichler1988}}]
    \Frm is dually equivalent to \LPries.
\end{theorem}

\begin{remark}
    The functors $\functor X : \Frm \to \LPries$ and $\functor{D} : \LPries \to \Frm$ establishing Pultr--Sichler duality are the restrictions of the functors establishing Priestley duality. We recall that the {\em Priestley space} of a frame $L$ is the set $X_L$ of prime filters of $L$ ordered by inclusion and topologized by the subbases $\{\varphi(a) \mid a \in L\} \cup \{\varphi(a)^c \mid a \in L\}$, where $\varphi$ is the Stone map given by $\varphi(a) = \{x \in X_L \mid a \in x\}$ for each $a \in L$. The functor $\functor X$ sends a frame $L$ to its Priestley space $X_L$ and a frame homomorphism $h : L \to M$ to the L-morphism $h^{-1} : X_M \to X_L$. The functor $\functor{D}$ sends an L-space $X$ to the frame $\clopup(X)$ of clopen upsets of $X$ and an L-morphism $f : X \to Y$ to the frame homomorphism $f^{-1} : \clopup(Y) \to \clopup(X)$.
\end{remark}
We next characterize Priestley spaces of spatial frames. 

\begin{definition}
    Let $X$ be an L-space. 
    \begin{enumerate}
        \item The set $Y := \{y \in X \mid \downset y \text{ is clopen}\}$ is called the \emph{spatial part of $X$}. 
        \item We call $X$ an \emph{SL-space} if 
        $Y$ is dense in $X$.
        \item Let \SL be the full subcategory of \LPries consisting of SL-spaces.
    \end{enumerate}
\end{definition}

Let $L$ be a frame. Recall (see, e.g., \cite[p.~15]{PicadoPultr2012}) that a \emph{point} of $L$ is a completely prime filter, and that the set $pt(L)$ of points of $L$ is topologized by $\{\varphi(a) \cap pt(L) \mid a \in L\}$. We will refer to $pt(L)$ as the \emph{space of points of $L$}.

\begin{remark}
    Let $X$ be an L-space and $Y$ the spatial part of $X$. 
    \begin{enumerate}[ref=\theremark(\arabic*)]
        \item We view $Y$ as a topological space,
        where $V\subseteq Y$ is open iff $V = U \cap Y$ for some $U \in \clopup(X)$. If $X$ is the Priestley space of a frame $L$, then the spatial part $Y$ of $X$ is exactly
    the space of points of $L$ (see, e.g., \cite[Lemma 4.1]{BezhanishviliMelzer2022b}).
        \item If $X$ is an SL-space, then for each $U \in \clopup(X)$ we have $\cl(U \cap Y) = U$. Therefore, the assignment $U \mapsto U \cap Y$ is an isomorphism from the poset of clopen upsets of $X$ to the poset of open sets of $Y$. This will be utilized in what follows.  \label[remark]{rem:iso opens Y}
    \end{enumerate}
\end{remark}

\begin{theorem}[{\cite[Section 4]{BezhanishviliMelzer2022b}}] \label{thm: SL dualities}
\SL is equivalent to \Sob and dually equivalent to \SFrm.
\end{theorem}

\begin{remark}
\begin{enumerate}[ref=\theremark(\arabic*)]
    \item[]
    \item The dual equivalence between \SFrm and \SL is obtained by restricting the functors establishing Pultr--Sichler duality. 
    \item The equivalence between \SL and \Sob is obtained as follows. Let $\functor{Y} : \LPries \to \Sob$ be the functor that sends an L-space $X$ to to its spatial part $Y$, and an L-morphism $f : X_1 \to X_2$ to its restriction $g:Y_1\to Y_2$.
    Then $\functor{Y}$  
    restricts to an equivalence between  $\SL$ and $\Sob$ (see, e.g., \cite[Corollary 4.19]{BezhanishviliMelzer2022b}).
    \label[remark]{rem:Y functor}
    \item As an immediate consequence of \cref{thm: SL dualities}, we obtain the well-known duality between \SFrm and \Sob.
    \end{enumerate}
\end{remark}

We now turn our attention to Priestley spaces of continuous frames.

\begin{definition} \label{def:kernel}\label{def:packed}
Let $X$ be an L-space. 
\begin{enumerate}
\item For $U,V\in{\sf ClopUp}(X)$, define $V \ll U$ provided for each open upset $W$ of $X$ we have $U \subseteq \cl W$ implies $V \subseteq W$.
\item For $U \in {\sf ClopUp}(X)$, define the \emph{kernel of $U$} as $$\ker U = \bigcup\{V \in {\sf ClopUp}(X) \mid V \ll U\}.$$ 
\item We call $X$ a \emph{continuous L-space} if $\ker U$ is dense in $U$ for each $U \in \clopup(X)$.
\item An L-morphism $f : X_1 \to X_2$ is \emph{proper} if $f^{-1}(\ker U) \subseteq \ker f^{-1}(U)$ for all $U \in \clopup (X_2)$.
\item Let \CL be the category of continuous L-spaces and proper L-morphisms.
\end{enumerate}
\end{definition}
\begin{theorem}[{\cite[Section 5]{BezhanishviliMelzer2022b}}] 
    \CL is equivalent to \LCSob and dually equivalent to \CFrm.
    \label{hl-frames} \label{hl-spaces}
\end{theorem}

\begin{corollary} [Hofmann--Lawson duality]
    \CFrm is dually equivalent to \LCSob.
\end{corollary}

We thus arrive at the following diagram which commutes up to natural isomorphism, where $\leftrightarrow$ represents equivalence.

\begin{figure}[H]
\begin{center}
    \begin{tikzcd}[ampersand replacement=\&, column sep=0em]
    \& \CFrm \ardual{dl} \ardual{dr}\\
    \CL \areqv{rr} \& \& \LCSob
\end{tikzcd}
\end{center}
\end{figure}

We next describe Priestley spaces of stably continuous and stably compact frames. For the next definition see \cite[Section 6]{BezhanishviliMelzer2022b}. The notion of L-compact first appeared in \cite[Section 3]{PultrSichler2000}.

\begin{definition} 
    \begin{enumerate}[ref=\thedefinition(\arabic*)]
    \item[]
    \item
    \begin{enumerate}[ref=\thedefinition(\arabic{enumi})(\alph*)]
        \item 
        An L-space $X$ is \emph{kernel-stable} if $\ker(U \cap V) = \ker U \cap \ker V$ for all $U,V \in \clopup(X)$, \label[definition]{def: kernel-stable}
        \item A \emph{stably continuous L-space} is a kernel-stable continuous L-space.
        \item Let \SCL be the full subcategory of \CL consisting of stably continuous L-spaces.
        \end{enumerate}
        \item
        \begin{enumerate}
        \item 
        An L-space $X$ is \emph{L-compact} if $X = \ker X$.
        \item A \emph{stably compact} L-space is an L-compact stably continuous L-space.
        \item Let \SKL be the full subcategory of \SCL consisting of stably compact L-spaces.
        \end{enumerate}
    \end{enumerate}
\end{definition}

\begin{theorem}[{\cite[Section 6]{BezhanishviliMelzer2022b}}] \label{thm: stably dualities}
    \begin{enumerate}[ref=\thetheorem(\arabic*)]
        \item[]
        \item \SCL is equivalent to \SLCSp and dually equivalent to \SCFrm. 
        \item \SKL is equivalent to \SKSp and dually equivalent to \SKFrm.
    \end{enumerate}
\end{theorem}

As a consequence, we obtain the following well-known dualities for stably continuous frames:

\begin{corollary}[{\cite[Corollary VI-7.2]{Compendium2003}}]
    \begin{enumerate}[ref=\thecorollary(\arabic*)]
    \item[]
    \item \SCFrm is dually equivalent to \SLCSp.
    \item \SKFrm is dually equivalent to \SKSp. \label[corollary]{cor:tpl}
    \end{enumerate}
\end{corollary}
\begin{figure}[H]
\begin{center}
    \begin{tikzcd}[ampersand replacement=\&, column sep=0em]
    \& \SCFrm \ardual{dl} \ardual{dr}\\
    \SCL \areqv{rr} \& \& \SLCSp
    \end{tikzcd}\hspace{2em}
    \begin{tikzcd}[ampersand replacement=\&, column sep=0em]
    \& \SKFrm \ardual{dl} \ardual{dr}\\
    \SKL \areqv{rr} \& \& \SKSp
    \end{tikzcd}
\end{center}
\end{figure}

We conclude this section by describing Priestley spaces of compact regular frames. The next definition appeared in \cite[Section 3]{BezhanishviliGabelaiaJibladze2016} and \cite[Section 7]{BezhanishviliMelzer2022b}.

\begin{definition}
Let $X$ be an L-space. 
\begin{enumerate}
\item For $U,V\in{\sf ClopUp}(X)$, define $V \prec U$ provided $\downset V \subseteq U$.
\item For $U \in {\sf ClopUp}(X)$, define the \emph{regular part of $U$} as $$\reg U = \bigcup\{V \in {\sf ClopUp}(X) \mid V \prec U\}.$$ 
\item We call $X$ a \emph{regular L-space} if $\reg U$ is dense in $U$ for each $U \in \clopup(X)$.
\item We call $X$ a \emph{compact regular L-space} if $X$ is a regular L-space that is  L-compact. 
\item Let \KRL be the full subcategory of \LPries consisting of compact regular L-spaces.
\end{enumerate}
\end{definition}

\begin{remark}
    Every L-morphism between compact regular L-spaces is proper (see \cite[Theorem 7.18(2)]{BezhanishviliMelzer2022b}), and every compact regular L-space is a stably compact L-space (see \cite[Theorem 7.17]{BezhanishviliMelzer2022b}). Thus, \KRL is a full subcategory of \SKL. 
\end{remark}

We have the following categories of continuous L-spaces.
\begin{table}[H]
\begin{tabular}{lll}
    \toprule 
    \multicolumn{1}{c}{\bf Category} & \multicolumn{1}{c}{\bf Objects} & \multicolumn{1}{c}{\bf Morphisms} \\ \midrule
    \CL & continuous L-spaces & proper L-morphisms \\
    \SCL & stably continuous L-spaces & proper L-morphisms \\
    \SKL & stably compact L-spaces & proper L-morphisms\\
    \KRL & compact regular L-spaces & L-morphisms\\
    \bottomrule
\end{tabular}
\caption{Categories of continuous L-spaces.\label{table:ConLspaces}}
\end{table}

\begin{theorem}[{\cite[Section 7]{BezhanishviliMelzer2022b}}] \label{thm: KRL duality}
\KRL is equivalent to \KHaus and dually equivalent to \KRFrm.
\end{theorem}

\begin{corollary} [Isbell duality]
\KRFrm is dually equivalent to \KHaus.
\end{corollary}

\begin{figure}[H]
\begin{center}
    \begin{tikzcd}[ampersand replacement=\&, column sep=0em]
    \& \KRFrm \ardual{dl} \ardual{dr}\\
    \KRL \areqv{rr} \& \& \KHaus
\end{tikzcd}
\end{center}
\end{figure}

We thus have the diagram in \cref{diagram: equivalences}.

\begin{figure}[H]
\begin{center}
\begin{tikzcd}[ampersand replacement=\&, column sep={6em,between origins}, row sep=2em]
    \CFrm \ardual{d} \arrestrictright 
        \& \SCFrm \ardual{d} \arrestrictright 
        \& \SKFrm \ardual{d} \arrestrictright 
        \& \KRFrm \ardual{d}\\
    \CL \arrestrictright  \areqv{d}
        \& \SCL \arrestrictright \areqv{d}
        \& \SKL \arrestrictright \areqv{d}
        \& \KRL \areqv{d}\\
    \LCSob \arrestrictright 
        \& \SLCSp \arrestrictright 
        \& \SKSp \arrestrictright 
        \& \KHaus
\end{tikzcd}
\end{center}
\caption{Equivalences and dual equivalences between categories of continuous frames, continuous L-spaces, and locally compact sober spaces.}\label{diagram: equivalences}
\end{figure}

In what follows, we will obtain a similar picture of equivalences and dual equivalences when the above categories of continuous frames are replaced by the corresponding full subcategories of algebraic frames.

\section{Priestley duality for algebraic frames} \label{sec: alg frm}

In this section we describe algebraic frames in the language of Priestley spaces. We then connect the Priestley duals of algebraic frames with compactly based sober spaces to derive the well-known duality between \AlgFrm and \KBSob mentioned in \cref{frm-dualities-alg}.

Let $X$ be an L-space and $Y$ the spatial part of $X$. We recall (see \cite[Definition 5.2]{BezhanishviliMelzer2022}) that a closed upset $F$ of $X$ is a \emph{Scott upset} if $\min F \subseteq Y$, where $\min F$ is the set of minimal points of $F$. Equivalently, $F$ is a Scott upset of $X$ iff 
\begin{equation*}
F \subseteq \cl U  \implies F \subseteq U \hspace{2em}\mbox{ for each open upset $U$ of $X$} \tag{$\dagger$} \label{dagger}
\end{equation*}
(see \cite[Lemma 5.1]{BezhanishviliMelzer2022}). We denote the collection of all clopen Scott upsets of $X$ by $\clopsup(X)$.

\begin{definition} \label{def: bunched}
    Let $X$ be an L-space. 
    \begin{enumerate}
    \item For $U \in \clopup(X)$, define the \emph{core} of $U$ as
    \[
        \core U = \bigcup \{V \subseteq U \mid V \in \clopsup(X)\}.
    \]
    \item Call $X$ an \emph{algebraic} L-space if $\core U$ is dense in $U$ for each $U\in \clopup(X)$.
    \end{enumerate}
\end{definition}

\begin{lemma} \plabel{lem:core}
    Let $X$ be an $L$-space and $U,V \in \clopup(X)$. \begin{enumerate}[ref=\thelemma(\arabic*)]
    \item $\core U \subseteq \ker U \subseteq U$. \clabel{1}
    \item $U \subseteq V$ implies $\core U \subseteq \core V$. \clabel{2}
    \item If $X$ is an algebraic L-space, then $X$ is a continuous L-space. \clabel{3}
    \item $U$ is a Scott upset iff $\core U = U$. \clabel{4}
    \end{enumerate}
\end{lemma}

\begin{proof}
    (1) Suppose $x \in \core(U)$. Then there is $V\in\clopsup(X)$ such that $x \in V \subseteq U$. Let $W$ be an open upset such that $U \subseteq \cl W$. Then $V \subseteq \cl W$, so $V \subseteq W$ by (\ref{dagger}). Hence, $V \ll U$.  Therefore, $x \in \ker U$, and so $\core U \subseteq \ker U$. That $\ker U \subseteq U$ follows from \cite[Lemma 5.2(1)]{BezhanishviliMelzer2022b}. 

    (2) This is obvious from the definition of the core.
        
    (3) Let $U \in \clopup(X)$. Since $X$ is an algebraic L-space, $\core U$ is dense in $U$. Therefore, $\ker U$ is dense in $U$ by (1). Thus, $X$ is a continuous L-space.
    
    (4) First suppose that $U$ is a Scott upset. By (1), $\core U \subseteq U$. Since $U$ is a Scott upset, $U \subseteq \core U$. Thus, $\core U = U$. Conversely, suppose that $U = \core U$. Since $U$ is compact, there are clopen Scott upsets $V_1, \dots, V_n \subseteq U$ such that $U = V_1 \cup \dots \cup V_n$. Because a finite union of Scott upsets is a Scott upset, $U$ is a Scott upset.
\end{proof}

We next connect algebraic frames with algebraic L-spaces. Let $L$ be a frame, $X$ its Priestley space, and $a \in L$. To simplify notation, we write $\core(a)$ for $\core\varphi(a)$ and $\ker(a)$ for $\ker \varphi(a)$. 

\begin{lemma}[{\cite[Lemma 6.10]{BezhanishviliMelzer2022b}}]  \label{lem:compact}
    Let $L$ be a frame and $X_L$ its Priestley space. For $a \in L$, the following are equivalent.
\begin{enumerate}
    \item $a$ is compact.
    \item $\ker(a) = \varphi(a)$.
    \item $\varphi(a)$ is a Scott upset.
    \end{enumerate}
    In particular, $L$ is compact iff 
    $X_L$ is L-compact.
\end{lemma}

\begin{theorem} \plabel{thm:alg-hb}
    Let $L$ be a frame and $X_L$ its Priestley space. 
    \begin{enumerate}[ref=\thetheorem(\arabic*)]
        \item For $a\in L$, we have $a = \bigvee\{b \in K(L) \mid b \leq a\}$ iff $\core(a)$ is dense in $\varphi(a)$.
        \item $L$ is an algebraic frame iff $X_L$ is an algebraic L-space. \clabel{2}
    \end{enumerate}
\end{theorem}

\begin{proof}
    (1) It is well known (see, e.g., \cite[Lemma 2.3]{BezhanishviliBezhanishvili2008}) that \[
        \varphi\left(\bigvee S\right)=\cl\left(\bigcup\{\varphi(s) \mid s \in S\}\right)
    \]
    for each $S \subseteq L$. Therefore, by \cref{lem:compact} we have $a = \bigvee\{b \in K(L) \mid b \leq a\}$ iff 
    \[
        \varphi(a) = \cl \left( \bigcup \{\varphi(b)\in \clopsup(X_L) \mid\varphi(b) \subseteq \varphi(a) \}\right) = \cl (\core(a)).
    \] 
    
    (2) follows from (1).
\end{proof}

We now turn to morphisms between algebraic L-spaces.

\begin{definition} \label{def: L-coherent}
\begin{enumerate}
    \item[]
    \item We call an L-morphism $f : X_1 \to X_2$ between L-spaces \emph{coherent} if 
    \[
    f^{-1}(\core U) \subseteq \core f^{-1}(U) 
    \hspace{2em}\mbox{ for all }U\in{\sf ClopUp}(X_{2}).
    \]
    \item Let \AlgL be the category of algebraic L-spaces and coherent L-morphisms.
\end{enumerate}
\end{definition}

It is easy to see that the identity morphism is a coherent L-morphism and that the composition of two coherent L-morphisms is coherent. Therefore, \AlgL is indeed a category. We show that \AlgL is a full subcategory of \CL. For this we need the following two lemmas.

\begin{lemma} \plabel{lem:scott-ext}
    Let $X$ be a continuous L-space and $U \in \clopup(X)$. The following are equivalent. \begin{enumerate}[ref=\thelemma(\arabic*)]
    \item
    $\ker U = \core U$. \clabel{1} 
    \item $\core U$ is dense in $U$.
    \item For each $y \in U \cap Y$, there is $V \in \clopsup(X)$ such that $y \in V \subseteq U$. \clabel{needed}
    \item For each Scott upset $F \subseteq \ker U$, there is $V \in \clopsup(X)$ such that $F \subseteq V \subseteq U$. \clabel{3}
    \end{enumerate}
\end{lemma}
\begin{proof}
    (1)$\Rightarrow$(2) Since $X$ is a continuous L-space, $\ker U$ is dense in $U$. Therefore, $\ker U = \core U$ implies that $\core U$ is dense in $U$.
    
    (2)$\Rightarrow$(3) Suppose $y \in U \cap Y$.  Since $U = \cl(\core U)$, we have $y \in \cl (\core U) \cap Y$. Because $\core U$ is an open upset, $\cl (\core U) \cap Y = \core U \cap Y$ by \cite[Lemma 4.14(1)]{BezhanishviliMelzer2022b}. Therefore, $y \in \core U$, and so 
    there is $V\in\clopsup(X)$ such that $y \in V \subseteq U$.
    
    (3)$\Rightarrow$(4) Let $F \subseteq \ker U$ be a Scott upset. Let $y \in F \cap Y$. Then $y \in \ker U$, so $y \in U$ by \cref{lem:core-1}. Therefore, by (3), there is $V_y \in\clopsup(X)$ such that $y \in V_y \subseteq U$. Thus, \[
    F = \bigcup \{\upset y \mid y \in F \cap Y\} \subseteq \bigcup \{V_y \mid y \in F \cap Y\} \subseteq U.
    \]
    Because $F$ is closed, it is compact. Therefore, since a finite union of clopen Scott upsets is a clopen Scott upset, there is $V \in \clopsup(X)$ such that $F \subseteq V \subseteq U$.
    
    (4)$\Rightarrow$(1). By \cref{lem:core-1}, $\core U \subseteq \ker U$. For the reverse inclusion, it suffices to show that $V \ll U$ implies there is $W \in \clopsup(X)$ such that $V \subseteq W \subseteq U$. 
    Let $V \ll U$. 
    Then there is a Scott upset $F$ such that $V \subseteq F \subseteq U$ (see, e.g., \cite[Lemma 5.7]{BezhanishviliMelzer2022b}). But $U = \cl (\ker U)$, so $F \subseteq \ker U$ by (\ref{dagger}). Therefore, by (4), there is $W \in \clopsup(X)$ such that $F \subseteq W \subseteq U$, and hence $V \subseteq W \subseteq U$.
\end{proof}

\begin{lemma} \plabel{lem:proper-coherent}
    Let $f : X_1 \to X_2$ be an L-morphism between L-spaces. 
    \begin{enumerate}[ref=\thelemma(\arabic*)]
        \item If $f$ is proper and $X_1$ is an algebraic L-space, then $f$ is coherent.
        \item If $f$ is coherent and $X_2$ is an algebraic L-space, then $f$ is proper.
        \item If $X_1$ and $X_2$ are algebraic L-spaces, then $f$ is coherent iff $f$ is proper. \clabel{3}
    \end{enumerate}
\end{lemma}
\begin{proof}
    (1) Let $U \in \clopup(X_2)$. Then
    \begin{align*}
        f^{-1}(\core U) 
        &\subseteq f^{-1}(\ker U) &&\text{by \cref{lem:core-1}}\\
        &\subseteq \ker f^{-1}(U) && \text{since } f \text{ is proper}\\
        &= \core f^{-1}(U) &&\text{by \cref{lem:core-3,lem:scott-ext-1}}.
    \end{align*}
    (2) Let $U \in \clopup(X_2)$. Then
    
    \begin{align*}
        f^{-1}(\ker U) 
        &= f^{-1}(\core U) &&\text{by \cref{lem:core-3,lem:scott-ext-1}}\\
        &\subseteq \core f^{-1}(U) && \text{since } f \text{ is coherent}\\
        &\subseteq \ker f^{-1}(U) &&\text{by \cref{lem:core-1}}.
    \end{align*}
    
    (3) follows from (1) and (2).
\end{proof}

Putting \cref{lem:core-3,lem:proper-coherent-3} together, we obtain:

\begin{theorem} \label{thm: full sub}
\AlgL is a full subcategory of \CL.
\end{theorem}

We are ready to prove the first main result of this section.

\begin{theorem} \label{thm: AlgFrm dual AlgL}
    \AlgFrm is dually equivalent to \AlgL.
\end{theorem}

\begin{proof}
By \cref{rem: full sub alg}, \AlgFrm is a full subcategory of \CFrm. By \cref{thm: full sub}, \AlgL is a full subcategory of \CL. Thus, the result follows from \cref{hl-frames,thm:alg-hb-2}.
\end{proof}

Finally, we connect \AlgL with \KBSob.

\begin{lemma}
        Let $X$ be an SL-space, $Y$ its spatial part, and $U \subseteq X$. Then $U \in \clopsup(X)$ iff there is a compact open set $V$ of $Y$ such that $\cl V = U$. \label{lem:clopen-scott}
\end{lemma}

\begin{proof}
By \cite[Theorem 5.7]{BezhanishviliMelzer2022}, the poset of Scott upsets of $X$ is isomorphic to the poset of compact saturated sets of $Y$. The  isomorphism is obtained by sending a Scott upset $F \subseteq X$ to the compact saturated set $F \cap Y$, and a compact saturated set $K \subseteq Y$ to the Scott upset $\upset K$.

($\Rightarrow$) Suppose $U$ is a clopen Scott upset. Then $V := U \cap Y$ is a compact saturated subset of $Y$. 
Moreover, $V$ is an open subset of $Y$ since $U \in \clopup(X)$. Furthermore, $\cl V = U$ by \cref{rem:iso opens Y} because $X$ is an SL-space.

($\Leftarrow$) Suppose there is a compact open set $V$ of $Y$ such that $\cl V = U$. 
Then $\upset V$ is a Scott upset of $X$. Since $V$ is open and $X$ is an SL-space, there is $U' \in \clopup(X)$ such that $V = U' \cap Y$ and $\cl V = U'$ (see \cref{rem:iso opens Y}). Therefore, $U = \cl V = U'$, and so $U$ is a clopen upset of $X$.
Moreover, 
\[
U = \upset U = \upset \cl V = \cl \upset V = \upset V,
\]
where the third equality follows from \cite[Theorem 3.1.2]{Esakia2019} since $X$ is an Esakia space.
Thus, 
$U$ is a Scott upset.
\end{proof}

\begin{theorem} \label{thm: HB iff KB}
    Let $X$ be an SL-space and $Y$ its spatial part. Then $X$ is an algebraic L-space iff $Y$ is a compactly based sober space.
\end{theorem}

\begin{proof}
The spatial part of an L-space is always sober (see, e.g., \cite[Lemma 4.11]
{BezhanishviliMelzer2022b}). Therefore, it is sufficient to show that $X$ is an algebraic L-space iff $Y$ is compactly based.
    First suppose that $X$ is an algebraic L-space. Let $V \subseteq Y$ be open and $y \in V$. Set $U = \cl V$. Then $U$ is a clopen upset of $X$ by \cref{rem:iso opens Y}. Moreover, it follows from \cite[Lemma 4.14(2)]{BezhanishviliMelzer2022b} that $U \cap Y = \cl V \cap Y = V$, so $y \in U \cap Y$. 
    By \cref{lem:core-3,lem:scott-ext-needed}, there is $W \in\clopsup(X)$ such that $y \in W \subseteq U$. 
    Therefore, $y \in W \cap Y \subseteq U \cap Y = V$. By \cref{lem:clopen-scott}, $W\cap Y$ is a compact open subset of $Y$. Thus, $Y$ is compactly based.
    
    Conversely, suppose that $Y$ is compactly based and $U \in \clopup(X)$. Since $Y$ is locally compact, $X$ is a continuous L-space by \cref{hl-spaces}. Therefore, by \cref{lem:scott-ext-needed}, it suffices to show that for each $y \in U \cap Y$ there is $V \in \clopsup(X)$ such that $y \in V \subseteq U$. Because $U \cap Y$ is an open subset of $Y$ and $Y$ is compactly based, there is a compact open $K \subseteq Y$ such that $y \in K \subseteq U \cap Y$. Therefore, $\cl K \in \clopsup(X)$ by \cref{lem:clopen-scott}. Moreover, $y \in \cl K \subseteq \cl (U \cap Y) = U$. Thus, $X$ is an algebraic L-space.
\end{proof}

By \cref{thm: full sub}, \AlgL is a full subcategory of \CL. By \cref{rem: full sub kbsob}, \KBSob is a full subcategory of \LCSob. Thus, as an immediate consequence of \cref{hl-spaces,{thm: HB iff KB}}, we obtain:

\begin{corollary} \label{cor: AlgL equiv KBSob}
        \AlgL is equivalent to \KBSob.
\end{corollary}

Putting together \cref{thm: AlgFrm dual AlgL,cor: AlgL equiv KBSob}, we obtain \cref{frm-dualities-alg} that $\AlgFrm$ is dually equivalent to \KBSob.

\section{Priestley duality for arithmetic, coherent, and Stone frames} \label{sec 5}

In this final section we describe Priestley duals of arithmetic, coherent, and Stone frames. We also connect them
to stably compactly based, spectral, and Stone spaces, thus obtaining alternative proofs of 
\namecref{frm-dualities} \hyperref[frm-dualities]{\labelcref{frm-dualities}(2,3,4)}.

We conclude the paper by pointing out a connection to Priestley duality for bounded distributive lattices and Stone duality for boolean algebras.  

\subsection{Arithmetic frames}

We recall (see \cref{def: kernel-stable}) that an L-space $X$ is kernel-stable if $\ker(U \cap V) = \ker U \cap \ker V$ for all $U,V \in \clopup(X)$.

\begin{definition} \label{def: ArithL}
    \begin{enumerate}
    \item[]
    \item An \emph{arithmetic} L-space is a kernel-stable algebraic L-space.
    \item Let \ArithL be the full subcategory  of \AlgL consisting of arithmetic L-spaces.
    \end{enumerate}
\end{definition}

\begin{lemma} \label{lem:scottstable}
    Let $X$ be an algebraic L-space. Then $X$ is an arithmetic L-space iff $U \cap V \in \clopsup(X)$ for every $U,V \in \clopsup(X)$.
\end{lemma}

\begin{proof}
    For the left-to-right implication,
    let $U,V \in \clopsup(X)$. By \cref{lem:compact}, $\ker U=U$ and $\ker V=V$. Therefore, since $X$ is kernel-stable, $\ker (U\cap V) = \ker U \cap \ker V = U \cap V$. Thus, $U\cap V \in \clopsup(X)$ using \cref{lem:compact} again.
    
    For the right-to-left implication,
    suppose $U_1,U_2 \in \clopup(X)$. It suffices to show that $W \subseteq \ker U_1 \cap \ker U_2$ iff $W \subseteq \ker (U_1 \cap U_2)$ for each $W \in \clopup(X)$. 
    Since $W$ is compact, by \cref{lem:scott-ext-1} and the assumption that $V_1, V_2 \in\clopsup(X) \Rightarrow V_1 \cap V_2 \in\clopsup(X)$, 
    \begin{align*}
    W \subseteq \ker U_1 \cap \ker U_2
    &\iff W \subseteq \core U_1 \cap \core U_2 \\
    &\iff \exists V_1, V_2 \in\clopsup(X) : W \subseteq V_1 \subseteq U_1 \text{ and }W \subseteq V_2 \subseteq U_2 \\
        &\iff\exists V \in \clopsup(X) : W \subseteq V \subseteq U_1 \cap U_2\\
        &\iff W \subseteq \core (U_1 \cap U_2) \\
        &\iff W \subseteq \ker (U_1 \cap U_2). \qedhere
    \end{align*}
    \end{proof}

\begin{lemma} \label{lem: stably compactly based}
    Let $Y$ be a compactly based sober space. Then $Y$ is stably locally compact iff $Y$ is stably compactly based.
\end{lemma}

\begin{proof}
    The left-to-right implication 
    is trivial.
    For the right-to-left implication, 
    let $A, B \subseteq Y$ be compact saturated. Since $Y$ is compactly based, every compact saturated set is an intersection of compact open sets (see \cref{rem: comp sat}). Therefore, $A \cap B = \bigcap \mathcal F$, where 
    \[\mathcal F = \{U \cap V \mid U,V \text{ compact open with }A \subseteq U \text{ and } B \subseteq V\}.\]
    Since $Y$ is stably compactly based, $\mathcal F$ is closed under finite intersections. Thus, the Hofmann--Mislove Theorem (see, e.g., \cite[Corollary II-1.22.]{Compendium2003}) implies that $\bigcap \mathcal F$ is compact. Consequently, $A \cap B$ is compact.
\end{proof}

\begin{theorem} \label{thm:arithmetic}
    Let $L$ be an algebraic frame, $X_L$ its Priestley space, and $Y_L \subseteq X_L$ the spatial part of $X_L$. The following are equivalent.
    \begin{enumerate}
        \item $L$ is an arithmetic frame.
        \item $X_L$ is an arithmetic L-space.
        \item $Y_L$ is a stably compactly based space.
    \end{enumerate}
\end{theorem}
\begin{proof}
    Since $L$ is an algebraic frame, $X_L$ is an algebraic L-space by \cref{thm:alg-hb}, and hence $Y_L$ is a compactly based sober space by \cref{thm: HB iff KB}.

    (1)$\Leftrightarrow$(2) Suppose $L$ is an arithmetic frame. Let $\varphi(a),\varphi(b) \in \clopsup(X_L)$. Then $a, b \in K(L)$ by \cref{lem:compact}. Since $L$ is an arithmetic frame, $a \wedge b \in K(L)$. Therefore, $\varphi(a) \cap \varphi(b) = \varphi(a \wedge b)$ is a Scott upset,  again by \cref{lem:compact}. Thus, $X_L$ is an arithmetic L-space by \cref{lem:scottstable}.
    
    Conversely, suppose $X_L$ is an arithmetic L-space.
    Let $a,b \in K(L)$. By \cref{lem:compact}, $\varphi(a),\varphi(b)$ are clopen Scott upsets. By \cref{lem:scottstable}, $\varphi(a \wedge b) = \varphi(a) \cap \varphi(b)$ is a Scott upset. Therefore, $a \wedge b \in K(L)$, again by \cref{lem:compact}. Thus, $L$ is an arithmetic frame.
    
    (2)$\Leftrightarrow$(3) Since $X_L$ is an algebraic L-space, 
    $X_L$ is an arithmetic L-space iff $X_L$ is a stably continuous L-space by \cref{lem:scottstable}. But $X_L$ is a stably continuous L-space iff $Y_L$ is a stably locally compact space by \cite[Theorem 6.7]{BezhanishviliMelzer2022b}. However, since $Y_L$ is a compactly based sober space, $Y_L$ is stably locally compact iff $Y_L$ is stably compactly based by \cref{lem: stably compactly based}. Thus, $X_L$ is an arithmetic L-space iff $Y_L$ is a stably compactly based space.
\end{proof}

As a consequence of \cref{thm: AlgFrm dual AlgL}, \cref{cor: AlgL equiv KBSob}, and \cref{thm:arithmetic}, we arrive at the first main result of this section:

\begin{theorem} \label{cor: duality for ArithFrm}
    \ArithL is equivalent to \SKBSp and dually equivalent to \ArithFrm.
\end{theorem}

As a corollary we obtain \cref{frm-dualities-arith} that \ArithFrm is dually equivalent to \SKBSp.

\subsection{Coherent frames}

We next turn our attention to Priestley duals of coherent frames. Since coherent frames are exactly compact arithmetic frames, we obtain that Priestley duals of coherent frames are exactly arithmetic L-spaces that are L-compact (see \cref{lem:compact}). 
We then connect L-compact arithmetic L-spaces with spectral spaces to obtain the well-known duality between \CohFrm and \Spec discussed in \cref{frm-dualities-coh}.

\begin{definition} \label{def: CohL}
    \begin{enumerate}
    \item[]
    \item A \emph{coherent} L-space is an L-compact arithmetic L-space.
    \item Let \CohL be the full subcategory  of \ArithL consisting of coherent L-spaces.
    \end{enumerate}
\end{definition}

\begin{lemma}[{\cite[Lemma 6.15]{BezhanishviliMelzer2022b}}] \label{lem:compact=tight}
    Let $X$ be an SL-space and $Y$ its spatial part. Then $X$ is L-compact iff $Y$ is compact.
\end{lemma}

\begin{theorem} \label{thm:CohL-equivalences}
    Let $L$ be an algebraic frame, $X_L$ its Priestley space, and $Y_L$ the spatial part of $X_L$. The following are equivalent.
\begin{enumerate}
    \item $L$ is a coherent frame.
    \item $X_L$ is a coherent L-space.
    \item $Y_L$ is a spectral space.
\end{enumerate}
\end{theorem}
\begin{proof}
    (1)$\Leftrightarrow$(2) $L$ is a coherent frame iff $L$ is a compact arithmetic frame. By \cref{lem:compact,thm:arithmetic}, this is equivalent to 
    $X_L$ being a coherent L-space.
    
    (2)$\Leftrightarrow$(3) 
    By \cref{lem:compact=tight,thm:arithmetic}, $X_L$ is a coherent L-space iff $Y_L$ is a compact stably compactly based space, hence a spectral space.
\end{proof}

As a consequence of \cref{cor: duality for ArithFrm,thm:CohL-equivalences}, we obtain the second main result of this section:

\begin{corollary} \label{cor:cohfrm=CohL=spec}
    \CohL is equivalent to \Spec and dually equivalent to \CohFrm.
\end{corollary}

As a corollary we obtain \cref{frm-dualities-coh} that \CohFrm is dually equivalent to \Spec.

\subsection{Stone frames}

Finally, we describe Priestley duals of Stone frames. Stone frames are characterized by having enough complemented elements. In the language of Priestley spaces, complemented elements correspond to clopen upsets that are also downsets (see, e.g., \cite[Lemma 6.1]{BezhanishviliGabelaiaJibladze2016}). 

Let $X$ be a Priestley space. Following \cite[p.~377]{BezhanishviliGabelaiaJibladze2016}, we call a subset of $X$ a \emph{biset} if it is both an upset and a downset. Let $\clopbi(X)$ be the collection of clopen bisets of $X$. 

\begin{definition} \label{def: centered}
    Let $X$ be an L-space.
    \begin{enumerate}
    \item For $U \in \clopup(X)$, define the \emph{center of $U$} as 
    \[
        \cen U = \bigcup \{V \in \clopbi(X) \mid V \subseteq U\}.
    \]
    \item We call $X$ a \emph{zero-dimensional} L-space if for every $U \in \clopup(X)$ we have that $\cen U$ is dense in $U$.
    \item A \emph{Stone} L-space is an L-compact zero-dimensional L-space.
    \item Let \StoneL be the full subcategory of \LPries consisting of Stone L-spaces.
    \end{enumerate}
\end{definition}

\begin{remark}
    In \cite[Definition 6.2]{BezhanishviliGabelaiaJibladze2016}, the center of a clopen upset $U$ is called the biregular part of $U$.
\end{remark}

\begin{lemma} \plabel{lem:center}
    Let $X$ be an L-space and $U \in \clopup(X)$.
    \begin{enumerate}[ref=\thelemma(\arabic*)]
    \item $\cen U \subseteq \reg U$. \clabel{cen<reg}
    \item If $X$ is a zero-dimensional L-space, then $X$ is a regular L-space.
    \item If $X$ is a Stone L-space, then $X$ is a compact regular L-space. \clabel{stone=kr}
    \end{enumerate}
\end{lemma}
\begin{proof}
    (1) Suppose $x \in \cen U$. Then there is $V \in \clopbi(X)$ with $x \in V \subseteq U$. Therefore, ${\downarrow\uparrow} x \subseteq U$, so $x \in \reg U$ by \cite[Lemma 7.3(1)]{BezhanishviliMelzer2022b}. 
    
    (2) Suppose $U \in \clopup(X)$. Since $X$ is a zero-dimensional L-space, $\cen U$ is dense in $U$. But then $\reg U$ is dense in $U$ by (1). Thus, $X$ is a regular L-space.
    
    (3) This follows from (2) and \cref{lem:compact}. 
\end{proof}

As an immediate consequence, we obtain that \StoneL is a full subcategory of $\KRL$. We proceed to show that \StoneL is a full subcategory of \CohL.

\begin{lemma}
    Let $X$ be a Stone L-space.
    \begin{enumerate}[ref=\thelemma(\arabic*)]
        \item $\clopsup(X) = \clopbi(X)$. \label[lemma]{lem: clopsup=clobi}
        \item $\cen U = \core U$ for each $U \in \clopup(X)$. \label[lemma]{lem: cen=core}
    \end{enumerate}
\end{lemma}

\begin{proof}
(1) Since $X$ is a Stone L-space, it is a compact regular L-space by \cref{lem:center-stone=kr}. Now apply \cite[Lemma 7.15(4)]{BezhanishviliMelzer2022b}.
    
    (2) It suffices to show that for each clopen upset $V$ we have $V \subseteq \cen U$ iff $V \subseteq \core U$.
    Since $V$ is compact, finite unions of bisets are bisets, and finite unions of Scott upsets are Scott upsets, (1) implies
    \begin{align*}
        V \subseteq \cen U
        &\iff \exists W \in \clopbi(X) : V \subseteq W \subseteq U \\
        &\iff \exists W \in \clopsup(X) : V \subseteq W \subseteq U\\
        &\iff V \subseteq \core U. \qedhere
    \end{align*}
\end{proof}

\begin{theorem} \label{thm: full sub StoneL}
    \StoneL is a full subcategory of \CohL.
\end{theorem}
\begin{proof}
    Every Stone L-space is a coherent L-space by \cref{lem: cen=core}.  
    Also, since $\StoneL$ is a full subcategory of \KRL, 
    every L-morphism between Stone L-spaces is a proper L-morphism by \cite[Theorem 7.18(2)]{BezhanishviliMelzer2022b}. Therefore, every such morphism  
    is a coherent L-morphism by \cref{lem:proper-coherent-3}. Thus, \StoneL is a full subcategory of \CohL.
\end{proof}
In \cite[Theorem 6.3(1)]{BezhanishviliGabelaiaJibladze2016} it is shown that Priestley duals of zero-dimensional frames are exactly zero-dimensional L-spaces. We connect zero-dimensional L-spaces to zero-dimensional topological spaces.

\begin{lemma} \label{lem:clopbi=clop}
    Let $X$ be an L-space and $Y$ its spatial part.
    \begin{enumerate}[ref=\thelemma(\arabic*)]
    \item If $U \in \clopbi(X)$, then $U \cap Y$ is clopen in $Y$. \label[lemma]{lem:clopbi=clop-1}
    \item If $X$ is an SL-space and $V \subseteq Y$ is clopen, then there is $U \in \clopbi(X)$ such that $V = U \cap Y$. \label[lemma]{lem:clopbi=clop-2}
\end{enumerate}
\end{lemma}

\begin{proof}
    (1) This is immediate.

    (2) Let $V \subseteq Y$ be clopen. Since $V$ is open, there is $U \in \clopup(X)$ such that $V = U \cap Y$ and $\cl V= U$ (see \cref{rem:iso opens Y}). Similarly, since $V$ is closed, there is $W \in \clopup(X)$ such that $Y \setminus V = W \cap Y$ and $\cl (Y \setminus V) = W$. We have 
    \[
    U \cap W = \cl(V) \cap \cl(Y\setminus V) = \cl(V\cap(Y\setminus V)) = \varnothing,
    \] 
    where the second equality follows from \cite[Lemma 4.15]{BezhanishviliMelzer2022b} because $V,Y \setminus V$ are open in $Y$. Also, 
    \[
    U \cup W = \cl V \cup \cl(Y \setminus V) = \cl(V \cup (Y \setminus V)) = \cl Y = X.
    \]
        Thus, $U = X \setminus W$, and hence $U \in \clopbi(X)$.
\end{proof}

\begin{theorem} \label{lem:zerodim-sp}
    Let $X$ be an L-space and $Y$ its spatial part. 
    \begin{enumerate}[ref=\thetheorem(\arabic*)]
        \item If $X$ is a zero-dimensional L-space, then $Y$ is zero-dimensional.
        \item If $X$ is an SL-space, then $X$ is a zero-dimensional L-space iff $Y$ is zero-dimensional. \label[theorem]{lem:zerodim-sp-2}
\end{enumerate}
\end{theorem}

\begin{proof}
    (1) Suppose $X$ is a zero-dimensional L-space. Let $V \subseteq Y$ be open and $y \in V$. Then there is $U \in \clopup(X)$ such that $U \cap Y = V$. Since $\cen U$ is dense in $U$, we have $U \cap Y = \cl (\cen U) \cap Y = \cen U \cap Y$, where the last equality follows from \cite[Lemma 4.14(1)]{BezhanishviliMelzer2022b} because $\cen U$ is an open upset of $X$. 
    Therefore, there is $W \in\clopbi(X)$ such that $y \in W \subseteq U$. Thus, $y \in W \cap Y \subseteq V$ and $W \cap Y$ is clopen in $Y$ by \cref{lem:clopbi=clop-1}. 
    Consequently, $Y$ is zero-dimensional.
    
     (2) The left-to-right implication follows from (1). For the converse implication, suppose $Y$ is zero-dimensional. Let $U \in \clopup(X)$. Since $X$ is an SL-space, $U\cap Y$ is dense in $U$. Therefore, it suffices to show that $U \cap Y \subseteq \cen U$. Let $y \in U \cap Y$. Since $U \in \clopup(X)$, we have that $U \cap Y$ is open in $Y$. Because $Y$ is zero-dimensional, there is clopen $V \subseteq Y$ such that $y \in V \subseteq U \cap Y$. Since $V$ is clopen in $Y$, \cref{lem:clopbi=clop-2} implies that there is $W \in \clopbi(X)$ such that $V = W \cap Y$. Because $X$ is an SL-space, $\cl V = W$, and hence $y \in W \subseteq U$. Thus, $y \in \cen U$.
\end{proof}

\begin{corollary}\label{thm:zerodim}
    Let $L$ be a spatial frame, $X_L$ its Priestley space, and $Y_L$ the spatial part of $X_L$. The following are equivalent.
    \begin{enumerate}
        \item $L$ is a zero-dimensional frame.
        \item $X_L$ is a zero-dimensional L-space.
    \end{enumerate}
    If in addition $L$ is spatial, then \textup{(1)} and \textup{(2)} are equivalent to
    \begin{enumerate}[resume]
        \item $Y_L$ is a zero-dimensional space.
    \end{enumerate}
\end{corollary}

\begin{proof}
    That (1)$\Leftrightarrow$(2) follows from \cite[Theorem 6.3(1)]{BezhanishviliGabelaiaJibladze2016}, and that (2)$\Leftrightarrow$(3) follows from  \cref{lem:zerodim-sp-2}.
\end{proof}

\begin{corollary} \label{thm:StoneL=stone}
    Let $L$ be a frame, $X_L$ its Priestley space, and $Y_L$ the spatial part of $X_L$. The following are equivalent.
    \begin{enumerate}
        \item $L$ is a Stone frame.
        \item $X_L$ is a Stone L-space.
    \end{enumerate}
    If in addition $L$ is spatial, then \textup{(1)} and \textup{(2)} are equivalent to
    \begin{enumerate}[resume]
        \item $Y_L$ is a Stone space.
    \end{enumerate}
\end{corollary}

\begin{proof}
    (1)$\Leftrightarrow$(2) Apply \cref{lem:compact,thm:zerodim}. 
    
    (2)$\Leftrightarrow$(3) Apply \cref{lem:compact=tight,thm:zerodim}. 
\end{proof}

As an immediate consequence, we arrive at the last main result of this section:

\begin{corollary}
    \StoneL is equivalent to \Stone and dually equivalent to \StoneFrm.
\end{corollary}

\begin{proof}
    This follows from \cref{cor:cohfrm=CohL=spec,thm:StoneL=stone} and the observation that \StoneFrm, \StoneL, and \Stone are full subcategories of \CohFrm, \CohL, and \Spec, respectively (see \cref{rem: full sub stonefrm,thm: full sub StoneL,rem: full sub stone}). 
\end{proof}

\cref{frm-dualities-stone}, which states that \StoneFrm is dually equivalent to \Stone, is now immediate from the above corollary.

\begin{remark}
Let $L$ be a frame and $X_L$ its Priestley space. As we saw in this paper, there are various maps from the clopen upsets of $X_L$ to the open upsets of $X_L$, and the corresponding density conditions are responsible for various properties of $L$. In particular,
\begin{itemize}
\item $L$ is continuous iff $\ker U$ is dense in $U$ for each $U\in\clopup(X)$;
\item $L$ is algebraic iff $\core U$ is dense in $U$ for each $U\in\clopup(X)$;
\item $L$ is regular iff $\reg U$ is dense in $U$ for each $U\in\clopup(X)$;
\item $L$ is zero-dimensional iff $\cen U$ is dense in $U$ for each $U\in\clopup(X)$.
\end{itemize}
The strength of these properties of frames is then described by how these maps interact. For example, $\core U\subseteq\ker U$ for each $U\in\clopup(X)$ indicates that every algebraic frame is continuous, etc.
\end{remark}

To summarize, we have the following diagram, where we use the same notation as in the previous diagrams. An overview of the introduced categories of Priestley spaces is given in \cref{table: algebraic spaces}. The corresponding categories of
frames and spaces are described in \cref{table:alg frames,table:kb spaces}.

\begin{figure}[H] 
\centering
\begin{tikzcd}[ampersand replacement=\&, column sep={8em,between origins}, row sep=1em]
    \AlgFrm \ardual{r} \& \AlgL \areqv{r}  \& \KBSob  \\
    \ArithFrm  \ardual{r}\arrestrict \& \ArithL \areqv{r}\arrestrict \& \SKBSp \arrestrict\\
    \CohFrm  \ardual{r}\arrestrict \& \CohL \areqv{r}\arrestrict\& \Spec \arrestrict\\
    \StoneFrm  \ardual{r} \arrestrict \& \StoneL \areqv{r} \arrestrict\& \Stone\arrestrict
\end{tikzcd}
\caption{Equivalences and dual equivalences between various categories of algebraic frames, algebraic L-spaces, and compactly based sober spaces.\label{diagram 2}}
\end{figure}
\begin{table}[H]
\centering
\begin{tabular}{llll}
    \toprule
    \multicolumn{1}{c}{\bf Category} & \multicolumn{1}{c}{\bf Objects} & \multicolumn{1}{c}{\bf Morphisms} \\ \midrule
    \AlgL & algebraic L-spaces (Definition \ref{def: bunched}) & coherent L-morphisms (Definition \ref{def: L-coherent})\\
    \ArithL & arithmetic L-spaces (Definition \ref{def: ArithL}) & coherent L-morphisms \\
    \CohL & coherent L-spaces (Definition \ref{def: CohL}) & coherent L-morphisms \\
    \StoneL & Stone L-spaces (Definition \ref{def: centered}) & L-morphisms \\
    \bottomrule
\end{tabular}
\caption{Categories of algebraic L-spaces.\label{table: algebraic spaces}}
\end{table}

We conclude the paper by connecting the results obtained above with Priestley duality for bounded distributive lattices and Stone duality for boolean algebras.

\begin{remark}
    Let $L$ be a coherent frame, $X_L$ its Priestley space, and $Y_L$ the spatial part of $X_L$. As we pointed out in the Introduction, $K(L)$ is a bounded distributive lattice and $L$ is isomorphic to the frame of ideals of $K(L)$. Moreover, $P \mapsto P\cap K(L)$ is an isomorphism between $(Y_L,\subseteq)$ and the poset of prime filters of $K(L)$. However, the Priestley topology on $X_{K(L)}$ does {\em not} correspond to the restriction to $Y_L$ of the Priestley topology on $X_L$. 
    Indeed, let $\varphi_{K_L} : K(L) \to \clopup(X_{K(L)})$ be the Stone map. By identifying $X_{K(L)}$ with $Y_L$, we have $\varphi_{K(L)}(k) = \varphi(k) \cap Y_L$ 
    for each $k \in K(L)$. Since $X_{K(L)}$ has $\{\varphi_{K(L)}(k_1) \setminus \varphi_{K(L)}(k_2) \mid k_1,k_2 \in K(L) \}$ as a basis and $\clopsup(X)$ corresponds to $K(L)$ by \cref{lem:compact}, the topology on $Y_L$ corresponding to the Priestley topology on $X_{K(L)}$
    is generated by the basis $\{(A \setminus B) \cap Y_L \mid A, B \in \clopsup(X_L)\}$. 
    
    If $L$ is a Stone frame, then $K(L)$ is the set of complemented elements of $L$ (that is, $a\vee a^*=1$). 
    Therefore, $K(L)$ is a boolean algebra and $Y_L = \min X_L$ (see, e.g., \cite[Lemma 7.15(5)]{BezhanishviliMelzer2022b}). But again, the Stone topology on $X_{K(L)}$ is not the restriction to $Y_L$ of the Priestley topology on $X_L$. In fact, the restriction of the Priestley topology on $X_L$ is the discrete topology on $Y_L$ (because $\downset y = \{y\}$ is open for each $y \in Y_L$), while the topology on $Y_L$ corresponding to the Stone topology on $X_{K(L)}$ is generated by $\{A \cap Y_L \mid A \in \clopbi(X_L)\}$ (by the previous paragraph and \cref{lem: clopsup=clobi}).
\end{remark}

\bibliographystyle{abbrv}

\begin{thebibliography}{10}

\bibitem{AvilaBezhanishviliMorandiZaldivar2020}
F.~\'{A}vila, G.~Bezhanishvili, P.~J. Morandi, and A.~Zald\'{\i}var.
\newblock When is the frame of nuclei spatial: a new approach.
\newblock {\em J. Pure Appl. Algebra}, 224(7):106302, 20 pp., 2020.

\bibitem{AvilaBezhanishviliMorandiZaldivar2021}
F.~\'{A}vila, G.~Bezhanishvili, P.~J. Morandi, and A.~Zald\'{\i}var.
\newblock The frame of nuclei on an {A}lexandroff space.
\newblock {\em Order}, 38(1):67--78, 2021.

\bibitem{Banaschewski1980}
B.~Banaschewski.
\newblock The duality of distributive continuous lattices.
\newblock {\em Canadian J. Math.}, 32(2):385--394, 1980.

\bibitem{Banaschewski1981}
B.~Banaschewski.
\newblock Coherent frames.
\newblock In {\em Proceedings of the Conference on Topological and Categorical
  Aspects of Continuous Lattices, Lecture Notes in Math., Vol. 871}, pages
  1--11. Springer-Verlag, Berlin, 1981.

\bibitem{Banaschewski1989}
B.~Banaschewski.
\newblock Universal zero-dimensional compactifications.
\newblock In {\em Categorical topology and its relation to analysis, algebra
  and combinatorics ({P}rague, 1988)}, pages 257--269. World Sci. Publ.,
  Teaneck, NJ, 1989.

\bibitem{BanaschweskiMulvey1980}
B.~Banaschewski and C.~J. Mulvey.
\newblock Stone-\v{C}ech compactification of locales. {I}.
\newblock {\em Houston J. Math.}, 6(3):301--312, 1980.

\bibitem{BezhanishviliBezhanishvili2008}
G.~Bezhanishvili and N.~Bezhanishvili.
\newblock Profinite {H}eyting algebras.
\newblock {\em Order}, 25(3):211--227, 2008.

\bibitem{BezhanishviliGabelaiaJibladze2013}
G.~Bezhanishvili, D.~Gabelaia, and M.~Jibladze.
\newblock Funayama's theorem revisited.
\newblock {\em Algebra Universalis}, 70(3):271--286, 2013.

\bibitem{BezhanishviliGabelaiaJibladze2016}
G.~Bezhanishvili, D.~Gabelaia, and M.~Jibladze.
\newblock Spectra of compact regular frames.
\newblock {\em Theory Appl. Categ.}, 31:Paper No. 12, 365--383, 2016.

\bibitem{BezhGhilardi2007}
G.~Bezhanishvili and S.~Ghilardi.
\newblock An algebraic approach to subframe logics. {I}ntuitionistic case.
\newblock {\em Ann. Pure Appl. Logic}, 147(1-2):84--100, 2007.

\bibitem{BezhanishviliMelzer2022}
G.~Bezhanishvili and S.~D. Melzer.
\newblock Hofmann--{M}islove through the lenses of {P}riestley.
\newblock {\em Semigroup Forum}, 105(3):825--833, 2022.

\bibitem{BezhanishviliMelzer2022b}
G.~Bezhanishvili and S.~D. Melzer.
\newblock Deriving dualities in pointfree topology from {P}riestley duality.
\newblock {\em Appl. Categ. Structures}, 31(5):1--28, {\bibsortkey{2023a}}2023.

\bibitem{Bhattacharjee2018}
P.~Bhattacharjee.
\newblock Maximal {$d$}-elements of an algebraic frame.
\newblock {\em Order}, 36(2):377--390, 2019.

\bibitem{BirkhoffFrink1948}
G.~Birkhoff and O.~Frink, Jr.
\newblock Representations of lattices by sets.
\newblock {\em Trans. Amer. Math. Soc.}, 64:299--316, 1948.

\bibitem{BS81}
S.~Burris and H.~P. Sankappanavar.
\newblock {\em A course in universal algebra}, volume~78 of {\em Graduate Texts
  in Mathematics}.
\newblock Springer-Verlag, New York-Berlin, 1981.

\bibitem{Erne2009}
M.~Ern\'{e}.
\newblock Infinite distributive laws versus local connectedness and compactness
  properties.
\newblock {\em Topology Appl.}, 156(12):2054--2069, 2009.

\bibitem{Esakia1974}
L.~L. Esakia.
\newblock Topological {K}ripke models.
\newblock {\em Soviet Math. Dokl.}, 15:147--151, 1974.

\bibitem{Esakia2019}
L.~L. Esakia.
\newblock {\em Heyting algebras}, volume~50 of {\em Trends in Logic---Studia
  Logica Library}.
\newblock Springer, Cham, 2019.
\newblock Edited by G. Bezhanishvili and W. H. Holliday, Translated from the
  Russian by A. Evseev.

\bibitem{Compendium2003}
G.~Gierz, K.~H. Hofmann, K.~Keimel, J.~D. Lawson, M.~Mislove, and D.~S. Scott.
\newblock {\em Continuous lattices and domains}, volume~93 of {\em Encyclopedia
  of Mathematics and its Applications}.
\newblock Cambridge University Press, Cambridge, 2003.

\bibitem{GierzKeimel1977}
G.~Gierz and K.~Keimel.
\newblock A lemma on primes appearing in algebra and analysis.
\newblock {\em Houston J. Math.}, 3(2):207--224, 1977.

\bibitem{Hochster1969}
M.~Hochster.
\newblock Prime ideal structure in commutative rings.
\newblock {\em Trans. Amer. Math. Soc.}, 142:43--60, 1969.

\bibitem{HofmannKeimel1972}
K.~H. Hofmann and K.~Keimel.
\newblock {\em A general character theory for partially ordered sets and
  lattices}.
\newblock Memoirs of the American Mathematical Society, No. 122. American
  Mathematical Society, Providence, R.I., 1972.

\bibitem{HofmannLawson1978}
K.~H. Hofmann and J.~D. Lawson.
\newblock The spectral theory of distributive continuous lattices.
\newblock {\em Trans. Amer. Math. Soc.}, 246:285--310, 1978.

\bibitem{IberkleidMcGovern2009}
W.~Iberkleid and W.~W. McGovern.
\newblock A natural equivalence for the category of coherent frames.
\newblock {\em Algebra Universalis}, 62(2-3):247--258, 2009.

\bibitem{Isbell1972}
J.~R. Isbell.
\newblock Atomless parts of spaces.
\newblock {\em Math. Scand.}, 31:5--32, 1972.

\bibitem{Jakl2013}
T.~Jakl.
\newblock Some point-free aspects of connectedness.
\newblock Master's thesis, Charles University, 2013.

\bibitem{Johnstone1981}
P.~T. Johnstone.
\newblock The {G}leason cover of a topos. {II}.
\newblock {\em J. Pure Appl. Algebra}, 22(3):229--247, 1981.

\bibitem{Johnstone1982}
P.~T. Johnstone.
\newblock {\em Stone spaces}, volume~3 of {\em Cambridge Studies in Advanced
  Mathematics}.
\newblock Cambridge University Press, Cambridge, 1982.

\bibitem{Nachbin1949}
L.~Nachbin.
\newblock On a characterization of the lattice of all ideals of a {B}oolean
  ring.
\newblock {\em Fund. Math.}, 36:137--142, 1949.

\bibitem{PicadoPultr2012}
J.~Picado and A.~Pultr.
\newblock {\em Frames and locales}.
\newblock Frontiers in Mathematics. Birkh\"{a}user/Springer Basel AG, Basel,
  2012.

\bibitem{Priestley1970}
H.~A. Priestley.
\newblock Representation of distributive lattices by means of ordered {S}tone
  spaces.
\newblock {\em Bull. London Math. Soc.}, 2:186--190, 1970.

\bibitem{Priestley1972}
H.~A. Priestley.
\newblock Ordered topological spaces and the representation of distributive
  lattices.
\newblock {\em Proc. London Math. Soc.}, 24:507--530, 1972.

\bibitem{PultrSichler1988}
A.~Pultr and J.~Sichler.
\newblock Frames in {P}riestley's duality.
\newblock {\em Cahiers Topologie G\'{e}om. Diff\'{e}rentielle Cat\'{e}g.},
  29(3):193--202, 1988.

\bibitem{PultrSichler2000}
A.~Pultr and J.~Sichler.
\newblock A {P}riestley view of spatialization of frames.
\newblock {\em Cahiers Topologie G\'{e}om. Diff\'{e}rentielle Cat\'{e}g.},
  41(3):225--238, 2000.

\bibitem{Simmons1982}
H.~Simmons.
\newblock A couple of triples.
\newblock {\em Topology Appl.}, 13(2):201--223, 1982.

\end{thebibliography}

\end{document}